\documentclass[12pt]{amsart}

\usepackage{amsmath}

\usepackage{amsfonts,amssymb,amsthm,enumitem}
\usepackage{mathtools}
\usepackage{nicefrac,setspace}
\usepackage{xcolor}
\usepackage{hyperref}       % hyperlinks
\usepackage{url}            % simple URL typesetting
\usepackage{booktabs}       % professional-quality tables
\usepackage{amsfonts}       % blackboard math symbols
\usepackage{nicefrac}       % compact symbols for 1/2, etc.
\usepackage{microtype}      % microtypography

\newtheorem{theorem}{Theorem}
\newtheorem{claim}[theorem]{Claim}
\newtheorem*{theorem*}{Theorem}
\newtheorem*{claim*}{Claim}
\newtheorem*{remark*}{Remark}

\newtheorem*{lemma*}{Lemma}

\newtheorem{lemma}[theorem]{Lemma}

\renewcommand{\mod}{\mathop {\mathsf{mod}}}

\newcommand{\Z}{{\mathbb Z}}

\newcommand{\calX}{{\mathcal{X}}}

\newcommand{\N}{\mathbb{N}}
\newcommand{\eps}{\epsilon}

\newcommand{\R}{\mathbb{R}}

\newcommand{\F}{\mathbb{F}}
\newcommand{\E}{\mathop \mathbb{E}}

\newcommand{\nbits}{\{\pm 1\}}

\newcommand{\ip}[2]{\left\langle #1,#2\right\rangle}

\newcommand{\Expect}[1]{\mathop{\mathbb{E}}\left
	[ #1 \right ]}
\newcommand{\Ex}[2]{\mathop{\mathbb{E}}\displaylimits_{#1}\left
	[ #2 \right ]}

\title{Anti-concentration
in most directions}
  
\newcommand{\an}[1]{\textcolor{blue}{#1}}
 
\begin{document}

\author{Anup Rao}
\address{School of Computer Science, University of Washington}
\email{anuprao@cs.washington.edu}
\author{Amir Yehudayoff}
\address{Department of Mathematics, Technion-IIT}
\email{amir.yehudayoff@gmail.com}
\thanks{A.R.\ is supported by the National Science Foundation under agreement CCF-1420268. A.Y.\ is partially supported by ISF grant 1162/15.
This work was done while
the authors were visiting the Simons
Institute for the Theory of Computing.}

\begin{abstract}
We prove anti-concentration bounds 
for the inner product of two independent random vectors. 
%sampled from discrete cubes. 
For example, we show that if $A,B$ are subsets of the cube
$\{\pm 1\}^n$ with $|A| \cdot |B| \geq 2^{1.01 n}$, and $X \in A$ and $Y \in B$ are sampled independently and uniformly, 
then the inner product $\ip{X}{Y}$ takes on any fixed value with probability at most $O(\tfrac{1}{\sqrt{n}})$. 
Extending Hal{\'a}sz work,
we prove stronger bounds when the choices for $x$ are unstructured.
%The proof provides a general framework
%for establishing anti-concentration in discrete domains.
We also describe applications to communication complexity,
randomness extraction and additive combinatorics.
\end{abstract}

\maketitle

\section{Introduction}
Anti-concentration bounds establish that the distribution of outcomes of a random process is 
not concentrated in any small region. No single outcome is obtained too often.
Anti-concentration plays an important role in 
mathematics and computer
science. It is used in the  study of roots of random  polynomials~\cite{littlewood1943number},
 random matrix theory~\cite{kahn,tao2009inverse},
 communication complexity~\cite{chakrabarti2012optimal,vidick2012concentration,sherstov2012communication},
 quantum computation~\cite{Aaronson_2011},
and more.

A well-known example is the sum of independent identically distributed random
variables. If $Y \in \nbits^n$ is uniformly distributed, then 
the probability that $\sum_{j=1}^n Y_j$ takes any specific value
is at most $\binom{n}{\lceil n/2 \rceil}/2^n = O(\tfrac{1}{\sqrt{n}})$. 

This was studied and generalized 
by Littlewood and Offord~\cite{littlewood1943number},
Erd\"os~\cite{erdos1945lemma}, 
and many others. The classical Littlewood-Offord problem is about understanding the anti-concentration of the inner product $\ip{x}{Y} = \sum_{j=1}^n x_j Y_j$, for arbitrary $x \in \R^n$ and $Y \in \nbits^n$ chosen uniformly.
Various generalizations were 
studied by Frankl and Furedi~\cite{frankl1988solution}, 
Hal{\'a}sz~\cite{halasz1977estimates}
and others.

It is interesting to understand the most general conditions under which anti-concentration holds  (see~\cite{tao2010sharp}
and references within).
Anti-concentration certainly fails when the 
entries of $Y$ are not independent.
%entropy  of $Y$ is not full. 
We can, for example, sample $Y$ uniformly from the set of strings with exactly $\lceil n/2 \rceil$ entries that are $1$. Then $\sum_j Y_j$ is always the same, yet $Y$ has almost full entropy.

Can we somehow recover anti-concentration? 
%when the entropy is not full?
A natural setting is to consider the inner-product $\ip{X}{Y}$ of two independent variables. 
This indeed recovers anti-concentration,
as the following theorem shows.

\begin{theorem*}[Chakrabarti and Regev~\cite{chakrabarti2012optimal}]
	There is a constant $c >0$ such that if $A,B \subseteq \nbits^n$ are each of size at least $2^{(1-c)n}$ and $X \in A, Y \in B$ are sampled uniformly and independently, then 
	$$\Pr [|\ip{X}{Y}| \leq c \sqrt{n}  ] \leq 1-c.$$
\end{theorem*}

The theorem shows that $\ip{X}{Y}$
does not land in an interval of length 
much smaller than $\sqrt{n}$ with high probability. 
When studying anti-concentration, however,
what we are ultimately interested in is 
proving point-wise estimates.
We would like to control the
{\em concentration probability} 
$$\max_{k \in \Z}  \ \Pr[ \ip{X}{Y} = k] ;$$
see~\cite{tao2009inverse}
and references within.

Here we prove
a sharp bound on the concentration probability.
The bound holds for an overwhelming majority of 
directions $x$.

\begin{theorem}\label{thm:main}
	For every $\beta > 0$  and 
	$\delta > 0$,
	%$0 < \delta < \tfrac{\beta}{6}$, 
	there exists $C>0$ such that the following holds. If $B \subseteq \nbits^n$ is of size $2^{\beta n}$, then for all but $2^{n(1-\beta +\delta)}$ directions $x \in \nbits^n$, 
	$$\max_{k \in \Z} \Pr_Y[\ip{x}{Y}=k] \leq \tfrac{C}{\sqrt{n}}.$$ 
	% for every number $k$.  
\end{theorem}

	The theorem is sharp in the following sense.
	As mentioned above, the $O(\tfrac{1}{\sqrt{n}})$ bound is tight even when $A$ and $B$ are $\{\pm 1\}^n$.
	To see that the bound on the number of bad directions is sharp, observe that if $B \subset \nbits^n$ is the set of $y$'s where the first $(1-\beta)n$ coordinates are set to $1$ and $\sum_{j> (1-\beta)n} y_j = 0$, and $A \subset \nbits^n$ is the set of $x$'s where the last $\beta n$ coordinates are set to $1$ and $\sum_{j \leq (1-\beta)n} x_j = 0$, then 
	$$|B| \approx \tfrac{1}{\sqrt{n}} 2^{\beta n}
	\quad \& \quad |A| \approx \tfrac{1}{\sqrt{n}} 2^{(1-\beta )n},$$ yet 
	$\ip{x}{y}=0$ for every $x \in A$ and $y \in B$. There is no anti-concentration in this case,
	although $|A| \cdot |B|$ is roughly $2^n$.

Our proof builds a flexible framework for proving anti-concentration results in discrete domains. We use this framework to prove a more general theorem, stated below. The theorem extends results from~\cite{Erdos2,Sark,Stan,halasz1977estimates}
from the uniform distribution on $\{\pm 1\}^n$
to the case of non-uniform distributions.

When $Y$ is uniformly distributed, the additive structure of the entries in the direction vector  $x$  
controls anti-concentration~\cite{frankl1988solution}. If $x$ is unstructured, we get stronger anti-concentration bounds for $\ip{x}{Y}$. 
This idea is instrumental when analyzing random
matrices~\cite{kahn,tao2009inverse}.

We choose the direction $x$
from sets of the following form. 
We call a set $A \subset \Z^n$ a two-cube
if $A = A_1 \times A_2 \times \cdots A_n$,
where each $A_j = \{u_j,v_j\}$ consists of two distinct integers. 
The differences of $A$ are the numbers $d_j = u_j - v_j$
for $j \in [n]$.

The following theorem describes three cases 
that yield different anti-concentration bounds.
It shows that the additive structure of 
$A$ is deeply related to the bounds we obtain.
The less structured $A$ is, the stronger the bounds are.

The first bound in the theorem holds for arbitrary two-cubes.
The second bound holds
when all the differences $d_1,\ldots,d_n$ are distinct.
The third bound applies in more general settings where 
the set of differences is unstructured.
This is captured by the following definition. 
A set $S \subset \N$ of size $n$ is called a Sidon set, or a Golomb ruler, if the number of solutions to the equation
$s_1+ s_2 = s_3+ s_4$ for $s_1,s_2,s_3,s_4 \in S$
is $4\cdot {n \choose 2}+n$.
In other words, every pair of integers has a distinct sum. 
Sidon sets were defined by Erd\"os and Tur\'an~\cite{ET}
and have been studied by many others since.
We say that 
$S \subset \Z$ is a weak Sidon set
if the number of solutions to the equation
$\eps_1 s_1+ \eps_2 s_2 = \eps_3 s_3+\eps_4 s_4$ for 
$\eps_1,\ldots,\eps_4 \in \{\pm 1\}$ and
$s_1,\ldots,s_4 \in S$
is at most $100 n^2$. The number $100$ can be replaced by any other constant, we use it here just to be concrete.

\begin{theorem}\label{thm:main1}
For every $\beta > 0$ and $\delta >0$,
there exists $C>0$ such that the following holds. 
Let $A \subset \Z^n$ is a two-cube with differences
$d_1,\ldots,d_n$.
Let $B \subseteq \nbits^n$ be of size $2^{\beta n}$.
Let $Y$ be uniformly distributed in $B$.
\begin{enumerate}
\item For all but $2^{n(1-\beta +\delta)}$ directions $x \in A$, 
		$$\max_{k \in \Z} \Pr_Y\left[\ip{x}{Y} =k\right] \leq 
	C \sqrt{\ln(n)} n^{-0.5} .$$
\item If $d_1,\ldots,d_n$ are distinct, then
for all but $2^{n(1-\beta +\delta)}$ directions $x \in A$, 
		$$\max_{k \in \Z} \Pr_Y\left[\ip{x}{Y} =k\right] \leq 
	C \sqrt{\ln(n)} n^{-1.5} .$$
\item If $\{d_1,\ldots, d_n\}$ is a weak Sidon set of size $n$, then
for all but $2^{n(1-\beta +\delta)}$ directions $x \in A$, 
		$$\max_{k \in \Z} \Pr_Y\left[\ip{x}{Y} =k\right] \leq 
	C \sqrt{\ln(n)} n^{-2.5}.$$
	\end{enumerate}
\end{theorem}

The first bound in Theorem~\ref{thm:main1} nearly implies Theorem~\ref{thm:main}. It is weaker by a factor of $\sqrt{\ln(n)}$. However, it holds for all two-cubes, not just the hypercube $\{\pm 1\}^n$. 
The second bound almost
matches the sharp $O(n^{-1.5})$ bound 
that holds when $(u_j,v_j) = (j,-j)$ for each $j$ and $Y$ is uniform in the hypercube~\cite{Sark,Stan}.
We believe that the $\sqrt{\ln(n)}$ factor 
is not needed, but were not able to eliminate it.
%Suppose $A_j = \{j,-j\}$ for each $j$.
%In this case, $r_1(A) = 2n$, so we get anti-concentration in most directions to the tune of 
%\an{$O(n^{-1.5} \sqrt{\ln n})$
%with $\nu = \tfrac{1}{\ln(1/\rho)}$}. This bound is almost
%tight; the correct bound for this example is $O(n^{-1.5})$ even under the uniform distribution~\cite{Sark,Stan}. If instead the entries are geometrically 
%increasing $A_j = \{2^j,-2^j\}$, then $r_\ell(A) \leq (2\ell n)^\ell$, and setting $\ell = \Omega(n)$ gives exponentially small anti-concentration
%\an{with $\nu=1$}. 
%This bound holds for any two-cube,
%not just $\{\pm 1\}^n$}.
%\end{examples}
The theorem is, in fact, part of a more general phenomenon.
We postpone the full technical description to
Section~\ref{sec:UTC}.

\subsubsection*{The proof.}
Chakrabarti and Regev's proof  
uses the deep connection between the
discrete cube and Gaussian space.
They proved a geometric correlation inequality in 
Gaussian space, and then translated it to the cube.
Vidick~\cite{vidick2012concentration}
later simplified part of their argument,
but stayed in the geometric setting.
Sherstov~\cite{sherstov2012communication}
found a third proof
that uses Talagrand's inequality from convex geometry~\cite{talagrand1995concentration}
and ideas of Babai, Frankl and Simon
from communication complexity~\cite{babai1986complexity}.

There are several differences between our argument and the ones in~\cite{chakrabarti2012optimal,vidick2012concentration,sherstov2012communication}.
The main difference is that
the arguments from~\cite{chakrabarti2012optimal,vidick2012concentration,sherstov2012communication} are based, in one way or another,
on the geometry of Euclidean space.
The arguments in~\cite{chakrabarti2012optimal,vidick2012concentration}
prove a correlation inequality in Gaussian space and translate
it to the discrete world. 
%\an{deleted this sentence. not sure what you mean. Their theorem gives bounds in the discrete setting. Do you mean pointwise? Let us discuss if you want to add it back}
It seems that such an argument
can not yield point-wise bounds on the
concentration probability. 
A common ingredient in~\cite{chakrabarti2012optimal,sherstov2012communication}
is a step showing that every set of large enough measure
contains many almost orthogonal vectors
(this is called `identifying the hard core' in~\cite{sherstov2012communication}).
In~\cite{vidick2012concentration} this part of the argument
is replaced by a statement about a relevant matrix.
Our argument does not contain such steps.

Let us briefly discuss our proof at a high level. 
The proof is based on harmonic analysis (Section~\ref{sec:aC}).
The argument consists of two parts.
In the first part, we analyze the Fourier behavior
of $\ip{x}{Y}$
for $x$ fixed and $Y$ random.
We are able to identify a collection
of good $x$'s for which the Fourier spectrum
of the distribution of $\ip{x}{Y}$ decays rapidly.
In the second part,
we show that the number of bad $x$'s
is small by giving an explicit encoding of all of them.

Although the proofs of the two theorems
follow a similar strategy,
we were not able to completely merge them. 
%If we could prove Theorem \ref{thm:main1} with $\nu = 0$, this would imply Theorem \ref{thm:main}. 
The proof for the hypercube (Theorem~\ref{thm:main})
is carried in Section~\ref{sec:HC}.
The proof for general two-cubes (Theorem~\ref{thm:main1})
is given in Section~\ref{sec:UTC}.

\begin{remark*} 
Theorem \ref{thm:main1} can be used as a black box to prove the same bounds when $A_1,\dotsc, A_n$ are pairs of real numbers. To see this, think of the relevant real numbers as a finite dimensional vector space
over the rationals.
\end{remark*}

\newpage

\subsection*{Applications}

\subsubsection*{Communication Complexity}

Chakrabarti and Regev's main motivation was understanding
the randomized communication 
complexity of the gap-hamming problem.
The gap-hamming problem was introduced
by Indyk and Woodruff in the context of streaming algorithms~\cite{indyk2003tight}.
Proving lower bounds on its communication
complexity was a central open problem for almost ten years,
until Chakrabarti and Regev solved it~\cite{chakrabarti2012optimal}.
Vidick~\cite{vidick2012concentration} and Sherstov~\cite{sherstov2012communication}
later simplified the proof.

Our results also imply the lower bound for the randomized
communication complexity of the gap-hamming problem
(see e.g.~\cite{sherstov2012communication}).
As opposed to~\cite{chakrabarti2012optimal,vidick2012concentration,sherstov2012communication}, the proof presented here lies entirely in the discrete domain.
The underlying ideas may therefore be 
of independent interest.
%For example, we get the following lower bound
%that does not seems to follow from previous works.
%For $x,y \in \{\pm 1\}^n$,
%consider the `exact gap orthogonality' promise problem:
%\begin{align*}
%\ego_n(x,y) = \begin{cases}
%1 & \ip{x}{y} = 0 ,\\
%0 & |\ip{x}{y}| \geq \sqrt{n} , \\
%\star & \text{otherwise.}
%\end{cases}
%\end{align*}
%This problem is interesting only for even $n$.
%For odd $n$'s, a protocol that outputs $1$
%is always correct (this is promise problem).
%
%\begin{theorem}
%\label{thm:REGO}
%The randomized communication complexity
%of $\ego_n$ for even $n$
%with error $\tfrac{1}{n}$ is at least $\Omega(n)$.
%\end{theorem}
%
%Standard error reduction yields an 
%$\Omega(\tfrac{n}{\ln n})$ lower
%bound on the randomized communication 
%complexity of $\ego_n$ when the error is $\tfrac{1}{3}$.
%
%

\subsubsection*{Pseudorandomness}

Randomness is a computational resource~\cite{trevisan2000extracting}.
There are many sources of randomness,
and some of them are {\em weak} or imperfect.
Randomness extractors allow to use weak sources
of randomness as if they were perfect.

The study of randomness extractors is 
about constructing explicit maps
that transform weak sources of randomness
to almost uniform outputs.
The main goal 
is generating a uniform output
in the most general scenario possible.
This often requires ingenious constructions.

The scenario described above fits 
nicely in the context of
{\em two-source extractors}.
A two-source extractor maps two independent
random variables $X$ and $Y$ with significant min-entropy
to a single almost uniform output.

Chor and Goldreich~\cite{chor1988unbiased} used
Lindsey's lemma to show that
inner product modulo two is a two-source extractor.
Namely, the distribution of the bit $\ip{X}{Y} \ \mathsf{mod} \ 2$
is close to uniform as long as $|A| \cdot |B| \gg 2^{n}$.
Bourgain~\cite{bourgain2005more},
Raz~\cite{raz2005extractors} and
Chattopadhyay and Zuckerman~\cite{chattopadhyay2016explicit}
constructed two-source extractors with much better
parameters. In our work, we study a related but somewhat
different question.

The high-level suggestion is to 
investigate other pseudorandom properties satisfied by 
 known extractors. We already know that inner product is an excellent 
two-source extractor. Our work shows that the
inner product is anti-concentrated over the integers.  Theorems~\ref{thm:main}
and~\ref{thm:main1}, in fact, imply a stronger statement.
It is the analog of ``strong'' extraction in extractor lingo.
Not only is $\ip{X}{Y}$ anti-concentrated,
but an overwhelming majority of fixings $X=x$ lead to $\ip{x}{Y}$  being anti-concentrated.

\subsubsection*{Additive Combinatorics}

Additive combinatorics studies the behavior
of sets under algebraic operations~\cite{Taovu}.
It has many deep results,
and connections to other areas of mathematics,
as well as many applications in computer science.
Our main result can be interpreted
as showing that Hamming spheres are far from being sum-sets. Our results give quantitative bounds on the size of the intersection of any Hamming sphere with a sum-set.

Replace $\{\pm 1\}$ by the field $\F_2$
with two elements.
The sum-set of $A \subseteq \F_2^n$
and $B \subseteq \F_2^n$ is
$$A+B = \{ x + y : x \in A , y \in B\}.$$
If $X$ and $Y$ are sampled uniformly at random from $A$ and $B$, then $X+Y$ is supported on $A+B$.

The cube $\F_2^n$ is endowed with a natural metric---the Hamming distance $\Delta(x,y)$.
The sphere around $0$ is the collection
of all vectors with a fixed number of ones in them
{(a.k.a.\ a slice)}.
The inner product $I = \sum_j (-1)^{X_j} (-1)^{Y_j}$
is similar to the inner product studied above
(here $X_j,Y_j \in \{0,1\}$).
{The inner product is related to the Hamming distance
by} $I(X,Y) = n- 2 \Delta(X,Y)$. We saw that if $|A| \cdot |B| > 2^{1.01n}$, then $I$ is anti-concentrated.
We can conclude that the distribution of the Hamming
distance of $X+Y$ is anti-concentrated.
The set $A+B$ is far from any sphere.

\section{Harmonic Analysis}
\label{sec:aC}
We are interested in proving
anti-concentration for integer-valued random variables. Harmonic analysis is a natural framework 
for studying such random variables~\cite{halasz1977estimates}.

Let $Y$ be distributed in $\{\pm 1\}^n$.
Let $x \in \Z^n$ be a direction. 
Let $\theta$ be uniformly distributed in $[0,1]$,
independently of $Y$.
The idea is to use
\begin{align*}
\Pr_Y[\ip{x}{Y} = k] &= \Ex{Y}{ \Ex{\theta}{  \exp(2\pi i \theta \cdot (\ip{x}{Y}-k)) }}
\end{align*}
to bound
\begin{align}
\notag
\max_{k \in \Z} \Pr_Y[\ip{x}{Y} = k] \leq  \Ex{\theta}{ \Big |\Ex{Y}{ \exp(2\pi i \theta \cdot\ip{x}{Y})} \Big | } . \tag{$\star$}
\end{align}
This inequality is useful for two reasons.
First, the left hand side is a maximum over $k$,
while the right hand side is not.
So, there is one less quantifier to worry about.
Secondly, the right hand side lives in the Fourier
world, where it is easier to argue about
the underlying operators. 
For example, when the coordinates of $Y$
are independent, the expectation over $Y$
breaks into a product of $n$ simple terms.

\section{The Hypercube}
\label{sec:HC}

We start by considering directions in the hypercube $\{\pm 1\}^n$.
The following theorem controls the Fourier coefficients in most 
directions.

\begin{theorem} \label{thm:tech} 
For every $\beta >0$
and $\delta>0$,
there is $c >0$ so that the following holds.
Let $B \subseteq \nbits^n$ be of size $2^{\beta n}$.
For each $\theta \in [0,1]$, 
for all but $2^{n(1-\beta+\delta)}$ directions
$x \in \{\pm 1\}^n$,
\begin{align*}
\left |\Ex{Y}{ \exp(2\pi i \theta \cdot\ip{x}{Y})} \right | > 
 2 \exp (- c n \sin^2(4 \pi \theta )) %+ \exp(-c n)
 \end{align*} 
\end{theorem}

The rest of this section is devoted for proving
this theorem
(the proof appears in Section~\ref{sec:mainPf}).

\begin{proof}[Proof of Theorem~\ref{thm:main}] 
Theorem~\ref{thm:tech} promises that for each $\theta \in [0,1]$,
the size of
		\begin{align*}
		A_\theta = \Big\{x \in \{\pm 1\}^n :
		\Big |\Ex{Y}{ \exp(2\pi i \theta \cdot\ip{x}{Y})} \Big | > 
		2 \exp (-c n \sin^2(4 \pi \theta ))  \Big\}
		\end{align*} 
is at most $2^{n(1-\beta+\delta)}$.
		For each $x$, define 
		$$S_x = \{\theta \in [0,1]: x \in A_\theta\}.$$ 
		Since 
		$$\E_x |S_x| = \E_\theta \tfrac{|A_\theta|}{2^n} 
		\leq 2^{n(-\beta+\delta)},$$  
		by Markov's inequality,
		the number of $x \in A$ for which $|S_x| > 2^{-\delta n}$ is at most $2^{(1-\beta + 2\delta)n}$. 
Fix $x$ such that $|S_x| \leq 2^{-\delta n}$.
Bound
\begin{align*}
\Pr_Y[\ip{x}{Y} = k] &\leq  \Ex{\theta}{\Big |\Ex{Y}{ \exp(2\pi i \theta \cdot\ip{x}{Y})} \Big | } \\
%&\leq   2^{-\delta n}+ 2 \Ex{\theta}{\exp(- c n \sin^2(4 \pi \theta))  } .
&\leq   2^{-\delta n} +  2 \int_{0}^{1} \exp(- c n \sin^2(4 \pi\theta)) \, \mathrm{d}\theta .
\end{align*}
%Theorem~\ref{thm:tech} and the union bound 
%imply that for all but $m 2^{n(1-\beta+\delta)}$ choices for $x$, %$|\widehat{f_x}(\ell)| \leq \exp(-\alpha \sin^2(\ell \tfrac{2 \pi}{m}))$,
%%. If $x$ does not belong to this bad set, we have
%%
%\begin{align*}
%|f_x(k)| 
%%& = \left | \tfrac{1}{\sqrt{m}}  \sum_{\ell \in \Z_m } \widehat{f}(\ell)  e^{i \tfrac{2 \pi}{m} \ell k}\right|\\
%& \leq \tfrac{1}{\sqrt{m}} \cdot \sum_{\ell  \in \Z_m} |\widehat{f_x}(\ell)|\\
%& \leq 2^{-\delta n} + \tfrac{1}{m} \sum_{\ell \in \Z_m} \exp(- c n \sin^2(\ell \tfrac{4 \pi}{m})).
%\end{align*}
Each term in the integral occurs eight times;
it circles the origin twice.
Using the inequality $\sin(\zeta) > \tfrac{\zeta}{\pi}$ for $0<\zeta<\tfrac{\pi}{2}$, we can bound it by
\begin{align*}
&\leq   2^{-\delta n} + 16 \int_{0}^{1/8} \exp(- 16 c n \theta^2) \, \mathrm{d}\theta \\
& \leq   2^{-\delta n} +  \frac{8}{\sqrt{n}} \cdot \int_{-\infty}^{\infty} \exp(- 16 c  \zeta^2) \, \mathrm{d}\zeta.
\end{align*}
The integral converges, so this quantity is at most $\tfrac{C
}{\sqrt{n}}$.
%
%Bound
%\begin{align*}
%\Pr_Y[\ip{x}{Y} = k] &\leq  \Ex{\theta}{ \left |\Ex{Y}{ \exp(2\pi i \theta \cdot\ip{x}{Y})} \right | } \\
%&\leq  2^{-\delta n} + 
%\Ex{\theta}{\exp(-c n \sin^2(4 \pi \theta )) + 2^{-\delta n}} \\
%&= 2^{1-\delta n} + \Ex{\theta}{ \exp (-c n \sin^2(4 \pi \theta )) } .
%\end{align*}
%\begin{align*}
%\Ex{\theta}{ \exp (-c n \sin^2(4 \pi \theta )) } 
%& = \int_0^1 \Pr_\theta[ \exp (-c n \sin^2(4 \pi \theta )) > t]
%\, \mathrm{d} t \\
%& = \int_0^1 \Pr_\theta[ c n \sin^2(4 \pi \theta ) < - \ln t]
%\, \mathrm{d} t \\
%\end{align*}
%
%
%

\end{proof}

\subsection{A Single Direction}
\label{sec:fourierbound}
In this section we analyze
the behavior of $\ip{x}{Y}$ for a single direction $x \in \Z^n$.
{We also focus on a single Fourier coefficient $\Ex{Y}{\exp({i \eta \ip{x}{Y}})}$ for a fixed
angle $\eta \in [0,2\pi]$.}

We reveal the entropy of $Y$ coordinate by coordinate.
To keep track of this entropy, define
the following functions $\gamma_1,\ldots,\gamma_n$
from $B = \text{supp}(Y)$ to $\R$.
For each $j \in [n]$, let
$$\gamma_j(y) = \gamma_j(y_{<j}) 
= \min_{\epsilon \in \nbits} \Pr[Y_j = \epsilon | Y_{<j}= y_{<j}].$$

To understand the interaction between $x$
and $y$, we use the following $n$ measurements.
For $j \in [n-1]$,
define $\phi_j(x,y)$ to be half of the phase of the complex number 
$$\Ex{Y_{>j}|Y_j=1,Y_{<j}=y_{<j}}{
\exp({i \eta \ip{x_{>j}}{Y_{>j}}})} \cdot \overline{\Ex{Y_{>j}|Y_j=-1,Y_{<j}=y_{<j}}{\exp({i \eta \ip{x_{>j}}{Y_{>j}}})}}.$$
This quantity is not defined when $\gamma_j(y)=0$.
% $\Pr[Y_j =1|Y_{<j}=y_{<j}]$ is $1$ or $0$. 
In this case, set $\phi_j(x,y)$ to be zero. 
Define $\phi_n(x,y)$ to be zero.
The number $\phi_j(x,y)$ is determined by $y_{< j}$ and $x_{>j}$.

In the following we think of $x$ as fixed,
and of $\gamma_j$ and $\phi_j$ as random variables
that are determined by the random variable $Y$.

\begin{lemma} \label{lemma:tech} 
For each $x \in \R^n$,
every random variable $Y$ over $\{\pm 1\}^n$,
and every angle $\eta \in \R$,
{\begin{align*}
	\left |\Ex{Y}{\exp({i \eta \ip{x}{Y}})} \right |^2 & \leq \Ex{Y}{
\prod_{i \in [n]} (1- \gamma_j \sin^2( \phi_j + x_j \eta ))}.
	\end{align*}}
\end{lemma}

\begin{proof}
The proof is by induction on $n$.
We prove the base case of the  induction and the inductive  step 
simultaneously. 
Express
	\begin{align*}
\left |\Ex{Y}{\exp(i \eta \ip{x}{Y})} \right |^2  &= \left|\Ex{Y_1}{ \exp(i \eta  x_1 Y_1) \cdot \Ex{Y_{>1}|Y_1}{\exp(i \eta  \ip{x_{>1}}{Y_{>1}})}}\right|^2 \\
& = \left|p_1 \exp(i \eta x_1) Z_1 + p_{-1} \exp(- i \eta  x_1 ) Z_{-1}\right|^2,
\end{align*}
where for $\epsilon \in \{\pm 1\}$, $$p_\epsilon = \Pr[Y_1 = \epsilon]
\qquad \& \qquad Z_{\epsilon} = \Ex{Y|Y_1 =\epsilon}{\exp(i \eta \ip{x_{>1}}{Y_{>1}})}.$$
When $n=1$, we have $Z_1=Z_{-1}=1$.	 
Rearranging,
\begin{align*} 
&\left|p_1 \exp(i \eta  x_1) Z_1 + p_{-1} \exp(-i \eta  x_1 ) Z_{-1}\right|^2 \\ 
%&=  (p_1 \exp(i \eta   x_1) Z_1 + p_{-1} \exp(-i \eta  x_1 ) Z_{-1})\\ &\cdot \overline{(p_1 \exp(i \eta x_1) Z_1 + p_{-1} \exp(-i \eta  x_1) Z_{-1})} \\
& = p_1^2 |Z_1|^2 + p_{-1}^2 |Z_{-1}|^2 + p_1 p_{-1} (Z_1 \overline{Z_{-1}} \exp(i2\eta  x_1)+ \overline{Z_1} Z_{-1} \exp(-i2\eta x_1)) \\
&= p_1^2 |Z_1|^2 + p_{-1}^2 |Z_{-1}|^2 + 2 p_1 p_{-1} |Z_1| |Z_{-1}|  \cos(2 \phi_1 + 2 x_1 \eta) .
\end{align*} 
The last equality holds by the definition of $\phi_1$.

{There are two cases to consider.
When $\cos(2 \phi_1 + 2 x_1 \eta) < 0$, we continue to bound
\begin{align*} 
& < p_1^2 |Z_1|^2 + p_{-1}^2 |Z_{-1}|^2  \\
& \leq (p_1 |Z_1|^2 + p_{-1} |Z_{-1}|^2)(1-\gamma_1) \\
& \leq (p_1 |Z_1|^2 + p_{-1} |Z_{-1}|^2)(1-\gamma_1 \sin^2(\phi_1 +x_1 \eta)) .
%& \leq \Ex{Y_1}{|Z_{Y_1}|^2} 
%\exp \left( - \gamma_1  \sin^2(\phi_1 + x_1 \eta) \right) .
\end{align*} 
Recall that $\gamma_1$ and $\phi_1$
do not depend on $Y$.
When $\cos(2 \phi_1 + 2 x_1 \eta) \geq 0$,
using the inequality $a^2 + b^2 \geq 2ab$, we bound}
\begin{align*}
&\leq p_1^2 |Z_1|^2 + p_{-1}^2 |Z_{-1}|^2 + p_1 p_{-1}(|Z_1|^2+  |Z_{-1}|^2)  \cos(2\phi_1 + 2x_1 \eta) \\
&= p_1 |Z_1|^2 (p_1 + p_{-1} \cos(2\phi + 2x_1 \eta)) + p_{-1} |Z_{-1}|^2 (p_{-1} + p_{1} \cos(2\phi_1 + 2x_1 \eta)) \\
& \leq (p_1 |Z_1|^2 + p_{-1} |Z_{-1}|^2)   (1-\gamma_1 + \gamma_1 \cos(2 \phi_1 + 2x_1 \eta)) \\
& = \Ex{Y_1}{|Z_{Y_1}|^2}  (1 -2 \gamma_1  \sin^2(\phi_1 + x_1 \eta)) .
%\\
%& \leq \Ex{Y_1}{|Z_{Y_1}|^2} 
%\exp \left( - \gamma_1  \sin^2(\phi_1 + x_1 \eta) \right) .
\end{align*}
When $n=1$, we have proved the base case of the induction.
When $n>1$, apply induction
on $|Z_\eps|^2$.
\end{proof}

\subsection{A Few Bad Directions} 
\label{sec:encoding}

Lemma~\ref{lemma:tech} suggests 
proving that the expression
$\sum_j \gamma_j \sin^2( \phi_j + x_j \eta )$
is typically large. Namely, we aim to show that there are usually many coordinates $j$ for which 
both $\gamma_j$ and $\sin^2(\phi_j +  x_j \eta )$ are bounded away from zero.
{Our approach is to explicitly encode the cases
where this fails to hold.}

Recall that $Y$ is uniformly distributed in a set $B$ of size $|B| = 2^{\beta n}$.
Let $1\geq \lambda>1/n$ be a parameter.  
Set $0 < \kappa < \tfrac{1}{2}$ and $1\geq \tau>0$ to be parameters satisfying the conditions
\begin{align}
\label{1}
H\left(\tfrac{1}{\log(1/\kappa)}\right) =  \tau + H\left(\tau \right)= \lambda , 
\end{align}
%\begin{align*}
%H\left(\tfrac{1}{\log(1/\kappa)}\right) = \lambda \qquad \&
%\qquad \tau + H\left(\tfrac{\tau}{\beta - 2\lambda}\right)= \frac{\lambda}{\beta - 2\lambda}, 
%\end{align*}
where $H$ is the binary entropy function: $$H(\xi) = \xi \log(1/\xi) + (1-\xi) \log(1/(1-\xi)).$$

The encoding is based on the following two sets:
$$J(y)  = J_{B,\kappa}(y) = \{ j \in [n] : \gamma_j(y) \geq \kappa\}$$ 
and
$$G(x,y) = G_{B,\kappa,\theta}(x,y) = \Big\{j \in J(y) : \sin^2(\phi_j(x,y) + x_j \eta ) \geq \tfrac{\sin^2(2\eta)}{4}\Big\}.$$
We start by showing that there 
are few $y$'s for which $|J(y)|$ is small.

\begin{lemma}\label{lemma:encoding1}
The number of $y \in B$ with $|J(y)| \leq n(\beta -3\lambda)$ is at most $2^{n(\beta - 2\lambda)}$.
\end{lemma}
\begin{proof}
If $3 \lambda > \beta$, 
%the number of $y$'s satisfying the condition is $0$, and 
the statement is trivially true. So, in the rest of the proof, assume that $3\lambda \leq \beta$.
Each  $y \in B$ with $|J(y)| \leq  n(\beta - 3 \lambda)$  can be uniquely  encoded by the following data:
\begin{itemize}
\item[--] An vector $q \in \nbits^{t}$ with
$t = \lfloor n(\beta - 3 \lambda) \rfloor$.
\item[--] A subset $S \subseteq [n]$ of size $|S| \leq \tfrac{n}{ \log(1/\kappa)}$.
\end{itemize} 
Let us describe the encoding. 
The vector $q$ encodes 
the values taken by $y$ in the coordinates $J(y)$. We do not encode $J(y)$ itself, only the values of $y$ in the coordinates corresponding to $J(y)$.  
The set $S$ includes $j \in [n]$ if and only if 
$\gamma_j(y) < \kappa$.
%$\Pr[Y_j = y_j|Y_{<j} = y_{<j}] < \kappa$. 
Each string $y \in B$ has probability at least $2^{-n}$.
This implies that
$\kappa^{|S|} \geq 2^{-n}$. 

We can reconstruct $y$ from $q$ and $S$ 
by iteratively computing $y_1$, then $y_2$, and so on, until we 
get to $y_n$. Whether or not $1 \in J(y)$ is determined even before we know $y$. If $1 \in J(y)$ then $q$ tells us
what $y_1$ is.
If $1 \not \in J(y)$ and $1 \in S$ then
$y_1$ is the least likely value between $\pm 1$.
If $1 \not \in J(y)$ and $1 \not \in S$
then $y_1$ is the more likely value.
Given the value of $y_1$,
we can continue in the same way to compute the rest of $y$.

The number of choices for $q$ is at most $2^{n(\beta - 3\lambda)}$. 
The number of choices for $S$ is at most 
%$$\sum_{0 \leq s \leq n/\log(1/\kappa)} \binom{n}{s} \leq 
$2^{n H(1/\log(1/\kappa))} = 2^{\lambda n}$.
\end{proof}

Next, we argue that there 
are few $x$'s for which there are many $y$'s with small $G(x,y)$.

\begin{lemma} \label{lemma:encoding2}
 The number of $x \in A$ for which $$\Pr_Y[ |G(x,Y)| \leq \tau n ] \geq 2^{-\lambda n}$$ is at most $2^{n(1-\beta+6\lambda)}$. 
\end{lemma}

\begin{proof}
The lemma is proved by double-counting
the edges in a bipartite graph.
Let ${\calX}$ be the set we are interested in:
	$$\calX = \big\{x : \Pr_Y[ |G(x,Y)| \leq \tau n ] \geq 2^{-\lambda n}\big\}.$$
The left side of the bipartite graph is $\calX$
and the right side is $B$.
Connect $x \in \calX$ to $y \in B$ by an edge
if and only if $G(x,y) \leq \tau n$.
Let  $E$ denote the set of edges in this graph.

First, we bound the number of edges from below.
The number of edges that
touch each $x \in \calX$ is at least
$2^{-\lambda n} |B|$.
It follows that
$$|E| \geq 2^{-\lambda n} \cdot |\calX| \cdot |B|.$$

Next, we bound the number of edges from above.
By Lemma~\ref{lemma:encoding1},
the number of 
$y \in B$ so that $|J(y)| \leq n(\beta -3\lambda)$
is at most $2^{-2\lambda n}|B|$.
We shall prove that the number of edges
that touch each $y$ with  $|J(y)| > n(\beta -3\lambda)$ is at most
$2^{n(1-\beta + 4 \lambda)}$.
It follows that 
$$|E| \leq 2^{-2\lambda n }\cdot |\calX|\cdot |B|
+ |B|\cdot 2^{n(1-\beta + 4 \lambda)}.$$
We can conclude that
\begin{align*}
2^{-\lambda n}\cdot  |\calX| \cdot |B| &\leq 2^{-2\lambda n}\cdot |\calX| \cdot |B|
+ |B| \cdot 2^{n(1-\beta + 4 \lambda)} \\
\Rightarrow |\calX| & \
%leq  2^{n(1-\beta + 4 \lambda)}/(2^{-\lambda} - 2^{-2\lambda}) 
\leq 2^{n(1-\beta + 6 \lambda)},
\end{align*}
since $\lambda n >1$.

It remains to fix $y$ so that
$|J(y)| > n(\beta -3\lambda)$ and bound its degree from above.
This too is achieved by an encoding argument.
Encode each $x$ that is connected
to $y$ by an edge using the following data:
	\begin{itemize}
		\item[--] A vector $q \in \nbits^{t}$
		with $t = \lfloor n(1-\beta+3\lambda) \rfloor$.
		\item[--] The set $G(x,y)$.
%		inside a universe of size at most $n(\beta - 3\lambda)$.
		\item[--] A vector $r \in \nbits^{s}$ with
		$s = \lfloor \tau n \rfloor$.
	\end{itemize}
Let us describe the encoding. 
The vector $q$ specifies the values of $x$ on coordinates not in $J(y)$.
%where $\gamma(y)_j < \kappa$. 
There are at most $n- n(\beta -3\lambda) = n(1-\beta + 3\lambda)$ such coordinates. 
The size of $G(x,y)$ is at most $\tau n$.
%specifies the subset of the coordinates where $\gamma(y)_j \geq \kappa$ for which
%    $$\sin^2(\phi(x,y)_j/2 + x_j \eta ) \geq  \sin^2(\eta).$$ 
The vector $r$ specifies the values of $x$ in the coordinates of $G(x,y)$, written in descending order. 
    
    The decoding of $x$ from $q,S$ and $r$ is done as follows.
    Decode the coordinates of $x$ in descending order from $n$ to $1$. 
    If $n \not \in J(y)$ then we read the value of $x_n$ from $q$. 
    If $n \in J(y)$ and $n  \in G(x,y)$, we decode $x_n$ by reading its value from $r$. If $n \in J(y)$ and 
   $n \notin G(x,y)$, 
  then 
  $$\sin^2(\phi_n(x,y)  + x_n \eta) \leq \tfrac{\sin^2(2\eta )}{4}.$$ 
   The number $\phi_n(x,y)$ does not depend on $x$.
   The following claim implies that there is at most one value of $x_n$ that satisfies this property.
   
   \begin{claim}
   \label{clm:OneGoodChoice}
   For all $\varphi \in \R$ and $u,v \in \Z$,
   \begin{align*}
   \max\{ |\sin(\varphi +\eta u)| , | \sin(\varphi + \eta v)| \}
\geq \tfrac{|\sin(\eta (u - v))|}{2}. 
   \end{align*}
\end{claim}

\begin{proof}
Consider the map 
$$\varphi\mapsto g(\varphi) =  
\max \{|\sin(\varphi +\eta u)| , |\sin(\varphi + \eta v)|\}.$$
The minimum of this map 
is attained when $$|\sin(\varphi +\eta u)| = |\sin(\varphi + \eta v)|.$$ This happens only when 
{$\varphi = -\tfrac{\eta (u + v)}{2} 
\ \mod \ \tfrac{\pi}{2} \Z$.
%for some integer $t$. 
By symmetry,}
% it is enough to consider $t \in \{0,1\}$, so we get
\begin{align*}
g(\varphi) &\geq \min \{g(-\eta (u + v)/2),g(-\eta (u + v)/2+\pi/2)\} \\
&\geq
|\sin(\eta (u-v)/2) \cdot \cos (\eta (u - v)/2)|
\\
&= \frac{|\sin(\eta (u-v))|}{2} . \qedhere
\end{align*}
%where we used the identity $\sin(\eta) \cos(\eta) = \sin(2\eta)/2$.
\end{proof}

The claim implies that we can indeed reconstruct $x_n$.
Given $x_n$, we can similarly reconstruct
$x_{n-1}$, since $\phi_{n-1}$ depends only on $y$ and $x_n$. Continuing in this way, we can reconstruct $x_{n-2},\dotsc,x_1$. The total number of choices for $q,S,r$ is at most 
$2^{n(1-\beta+3\lambda) + n H(\tau) + \tau n  }
= 2^{n(1-\beta + 4 \lambda)}$.\qedhere
%    \begin{align*}
%2^{n(1-\beta+3\lambda) + \tau n + n H(\tau/(\beta - 3\lambda))}
%\leq 2^{n(1-\beta + 4 \lambda)} .
%    \end{align*}

%    The decoding succeeds whenever  $J(y) \geq n(\beta - 3\lambda)$ and $G(x,Y) \leq \tau n$. By Lemma \ref{lemma:encoding1}, $$\Pr_Y[J(Y) \leq n(\beta - 3\lambda)] \leq \frac{2^{n(\beta -2\lambda)}}{2^{\beta n}} = 2^{-2\lambda n}$$ By the assumption on $x$, $\Pr_Y[G(x,Y) \leq \tau n] \geq 2^{-\lambda n}$. Thus, the probability that the decoding succeeds is at least  $2^{-\lambda n} -2^{-2\lambda n} \geq 2^{-2\lambda n}$. So the number of possible strings $x$ meeting these conditions is at most $2^{n(1-\beta + 4\lambda + 2\lambda)} = 2^{n(1-\beta +6 \lambda)}$.
    
\end{proof}

\subsection{Putting It Together}
\label{sec:mainPf}

\begin{proof}[Proof of Theorem \ref{thm:tech}]
Set $\lambda = \tfrac{\delta}{6}$.
By Lemma \ref{lemma:tech}, 
	\begin{align*}
	\left| \Ex{Y}{\exp(2 \pi i \theta \ip{x}{Y})} \right| &
 \leq 
 \sqrt{\Ex{Y}{\exp\left (- \sum_{j=1}^n \gamma_j \sin^2(\phi_j + 2 \pi \theta  x_j )\right)}.}
\end{align*}
%In the case we are now considering,
%$$G(x,y) = \Big\{j \in J(y) : \sin^2(\phi_j(x,y) + x_j \eta ) \geq \tfrac{\sin^2(4 \pi \theta )}{4}\Big\}.$$
Whenever $x$ is such that 
\begin{align}
\label{2}
\Pr_Y[ G(x,Y) \leq \tau n ] < 2^{-\lambda n},
\end{align}
we can bound 
$$ \Ex{Y}{\exp\left (- \sum_{j=1}^n \gamma_j \sin^2(\phi_j + 2 \pi \theta  x_j )\right)}
\leq \exp  (- \tfrac{ \kappa}{4} n \tau \sin^2(4 \pi \theta))  
% \sum_{j \in [n]} \sin^2(2\pi \theta (u_j-v_j))\Big) 
+ 2^{-\lambda n}.$$
Since
$\sqrt{a+b} \leq \sqrt{a}+\sqrt{b}$ for $a,b \geq 0$,
for such an $x$ we can bound
	\begin{align*}
	\left| \Ex{Y}{\exp(2 \pi i \theta \ip{x}{Y})} \right| &
\leq  \exp  (- \tfrac{ \kappa}{8} n \tau \sin^2(4 \pi \theta))  
% \sum_{j \in [n]} \sin^2(2\pi \theta (u_j-v_j))\Big) 
+ 2^{-\lambda n/2} \\
& \leq 2 \exp ( - c n \sin^2(4 \pi \theta) ) .
\end{align*}
Lemma~\ref{lemma:encoding2} promises that there are at most $2^{n(1-\beta+ \delta)}$ choices for $x$ that does not satisfy~\eqref{2}.

\end{proof}

\section{General Two-Cubes}
\label{sec:UTC}
Now we move to the setting
where the direction $x$ is chosen from 
an arbitrary two-cube $A \subset \Z^n$
with differences $d_1,\ldots,d_n$.
The way we measure the structure
of $A$ follows ideas of Hal{\'a}sz~\cite{halasz1977estimates}.
For an integer $\ell > 0$, define $r_\ell(A)$ to be
the number of elements $(\epsilon, j) \in \{\pm 1\}^{2\ell} \times [n]^{2\ell}$ that satisfy $$\epsilon_{1} \cdot d_{j_1} + \dotsb + \epsilon_{2\ell} \cdot d_{j_{2\ell}} = 0.$$
The smaller $r_\ell(A)$ is, the less structured $A$ is.

The theorem below shows that $r_\ell(A)$ allows us 
to control the concentration probability. More concretely, for $C>0$ and $\ell > 0$, define 
$$R_{C,\ell}(A) = 
 \frac{C^\ell r_\ell(A)}{n^{2\ell+ 1/2}} + \exp(-\tfrac{n}{C} ).$$
Define
$$R_{C}(A) = 
\inf \{  R_{C,\ell}(A) : \ell \in \N \}.$$
This is essentially the bound on the concentration
probability that Hal{\'a}sz obtained in~\cite{halasz1977estimates}
when $Y$ is uniform in $\{\pm 1\}^n$.
Our upper bounds are slightly weaker. Let 
$$\mu_C(A)
= \inf \Big\{
\mu  \in [0,1] : \exists
 \nu \in (0,1] \ \ 
   \mu^{(1+\nu)^2}
  \geq 3  \exp(-\tfrac{ \nu n}{C}) +
 \tfrac{R_{C}(A)}{50 \sqrt{\nu}} \Big\},$$
where we adopt the convention that the infimum of the empty set is~$1$. 
 Before stating the theorem, let us go over the three
 examples from Theorem~\ref{thm:main1}:
\begin{enumerate}
\item
For arbitrary $A$, 
since $r_1(A) \leq O(n^2)$, we get\footnote{Here
and below the big $O$ notation
hides a constant that may depend on $C$.}
$\mu_C(A) \leq O(\tfrac{\sqrt{\ln n}}{\sqrt{n}})$
with $\nu = \tfrac{1}{\ln(1/R_{C,1}(A))}$.
\item When all the differences are distinct,
since $r_1(A) \leq O(n)$, we get
$\mu_C(A) \leq O(n^{-1.5} \sqrt{\ln n})$
with $\nu = \tfrac{1}{\ln(1/R_{C,1}(A))}$.
\item When $\{\pm d_1,\ldots,\pm d_n\}$ is a Sidon set,
since $r_2(A) \leq O(n^2)$,
we get
$\mu_C(A) \leq O(n^{-2.5} \sqrt{\ln n})$
with $\nu = \tfrac{1}{\ln(1/R_{C,2}(A))}$.
\end{enumerate}
More generally, when $R_C(A)$ is bound
from below by some
polynomial in $\tfrac{1}{n}$
then $\mu_C(A)$ is at most $O(R_C(A) \sqrt{ \log (4/R_C(A))})$.

%\an{with $\nu=1$}. 

%
%\begin{theorem}\label{thm:main11}
%For every $\beta > 0$ and $\delta >0$,
%%and every $0< \nu \leq 1$,
%there exists $C>0$ such that the following holds. 
%		Let $A \subset \Z^n$ be a two-cube.
%Let $B \subseteq \nbits^n$ be of size $2^{\beta n}$.
%%Let $\ell > 0$ be an integer. 
%%\an{If $R_C(A) \leq \tfrac{1}{4}$ 
%%$$\rho =  \frac{C^\ell r_\ell(A)}{n^{2\ell+ 1/2}} + \exp(-\tfrac{n}{C} ) \leq \tfrac{1}{4}$$
%%and 
%%there are $\mu > 0$ and $0 < \nu \leq 1$ so that
%%\begin{align*}
%%\mu^{(1+\nu)^2}
%% & \geq 3  \exp(-\tfrac{ \nu n}{C}) +
%% \tfrac{\rho}{50 \sqrt{\nu}} ,
%% \end{align*}
%%then} 
%Then, for all but $2^{n(1-\beta +\delta)}$ directions $x \in A$, 
%		$$\max_{k \in \Z} \Pr_Y\left[\ip{x}{Y} =k\right] \leq 
%		\mu_C(A).$$
%		%\rho \sqrt{\ln(1/\rho)}.$$
%\end{theorem}

\begin{theorem}
\label{thm:mainTechU}
For every $\beta > 0$ and $\delta >0$, 
there is $C>0$ so that the following holds.
Let $B \subseteq \{\pm 1\}^n$ be of size $2^{\beta n}$.
Let $Y$ be uniformly distributed in $B$.
Let $A \subset \Z^n$ be a two-cube.
%Let $\ell \in \N$. 
%\an{If $R_C(A) 
%$$\rho = 
%\frac{C^\ell r_\ell(A)}{ n^{2\ell + 0.5} } +
%C \exp({-\tfrac{n}{C}}) 
%\leq \tfrac{1}{4}$
%and there are $\mu > 0$ and $0 < \nu \leq 1$ so that
%\begin{align}
%\mu^{(1+\nu)^2}
% & \geq 3  \exp(-\tfrac{ \nu n}{C}) +
% \tfrac{R_C(A)}{50 \sqrt{\nu}} ,
% \end{align}
%then} 
Then, for all but $2^{n(1-\beta+\delta)}$ directions $x \in A$,
\begin{align*}
\Ex{\theta}{ \left |\Ex{Y}{ \exp(2\pi i \theta \cdot\ip{x}{Y})} \right|}
\leq \mu_C(A) .
%  \rho \sqrt{\ln(1/\rho)} .
\end{align*}
\end{theorem}

Before moving on, we discuss a fourth
extreme example. 
When $A_j = \{2^j,-2^j\}$ for each $j \in [n]$,
we have $r_\ell(A) \leq (2\ell n)^\ell$.
In this case, setting $\ell = \Omega(n)$ gives exponentially small anti-concentration with $\nu=1$.
This result is trivial, but it illustrates that the mechanism
underlying the proof yields strong bound in
many settings.

By ($\star$) and the explanation above,
we see that
Theorem~\ref{thm:mainTechU} implies Theorem~\ref{thm:main1}.
The rest of this section is devoted to the proof
of Theorem~\ref{thm:mainTechU}.
The high-level structure of the proof is similar
to that of Theorem~\ref{thm:tech}.
However, there are several new technical challenges
that we need to overcome. 

The main technical challenge that needs to be overcome has to do with the definition of the set $G$.  
The $G$ defined in the previous section depends on the angle $\theta$.
This is problematic for the proof in the generality 
we are working with now.
So, we need to find a different set of \emph{good} coordinates, one that depends only on $x$ and $y$.
Our solution is based on 
the following claim, which quantifies the strict
convexity of the map $\zeta \mapsto \zeta^{1+\nu}$
for $\nu >0$. We defer the proof to Appendix~\ref{sec:techclaim}.

\begin{claim}\label{clm:MinPhi} 
For every $\kappa > 0$,
there is a constant $c_1 > 0$ so that the following holds.
For every random variable $W \in \nbits$ such that
$$\min \big\{\Pr[W=1], \Pr[W=-1]\big\} \geq \kappa,$$ 
every $\alpha_1
\geq 2 \alpha_{-1}\geq 0$ 
and every $0 < \nu \leq 1$, 
	$$\Expect{\alpha_W}^{1+\nu} \leq (1-c_1 \nu)  \Expect{\alpha_W^{1+\nu}}.$$
\end{claim}

%\an{The claim is most meaningful for $\nu >0$;
%the bound for $\nu=0$ is trivial.
%In a nutshell, this is the cause of
%the logarithmic factor loss in Theorem~\ref{thm:main1}.}
%

\subsection{A Single Direction}
\label{sec:fourierboundU}

The following lemma generalizes Lemma~\ref{lemma:tech}.
Recall the definition of $\gamma_j$, $\phi_j$
and $J(y)$ from Sections~\ref{sec:fourierbound}
and~\ref{sec:encoding}.

\begin{lemma} \label{lemma:techU} 
For every $\kappa > 0$,  
there is a constant $c_0 > 0$ so that the following holds. For every $0 < \nu\leq 1$, every angle $\eta \in \R$, every direction $x \in \Z^n$, and every random variable $Y$ over $\{\pm 1\}^n$,
\begin{align*}
	\left |\Ex{Y}{\exp({i \eta \ip{x}{Y}})} \right|^{1+\nu} & \leq \Ex{Y}{
	\prod_{j \in J} (1- c_0 \nu \sin^2( \phi_j + x_j \eta ) }.
\end{align*}
\end{lemma}

\begin{proof}
The proof is by induction on $n$. If $1 \notin J$, the proof holds by induction. The base case of $n=1$ is trivial.
So assume that $1 \in J$. 
Express
\begin{align*}
\Ex{Y}{\exp(i \eta \ip{x}{Y})}  
% &= 
%\Ex{Y_1}{ \exp(i \eta  x_1 Y_1) \cdot \Ex{Y|Y_1}{\exp(i \eta  \ip{x_{>1}}{Y_{>1}})}}  \\
& =  p_1 \exp(i \eta x_1) Z_1 + p_{-1} \exp(- i \eta  x_1 ) Z_{-1} ,
\end{align*}
where for $\epsilon \in \{\pm 1\}$, $$p_\epsilon = \Pr[Y_1 = \epsilon]
\qquad \& \qquad Z_{\epsilon} = \Ex{Y|Y_1 =\epsilon}{\exp(i \eta \ip{x_{>1}}{Y_{>1}})}.$$
When $n=1$, we have $Z_1=Z_{-1}=1$.	 
Using the definition of $\phi_1$,
\begin{align*} 
&\left|p_1 \exp(i \eta  x_1) Z_1 + p_{-1} \exp(-i \eta  x_1 ) Z_{-1}\right|^2 \\ 
%&=  (p_1 \exp(i \eta   x_1) Z_1 + p_{-1} \exp(-i \eta  x_1 ) Z_{-1})\\ &\cdot \overline{(p_1 \exp(i \eta x_1) Z_1 + p_{-1} \exp(-i \eta  x_1) Z_{-1})} \\
& = p_1^2 |Z_1|^2 + p_{-1}^2 |Z_{-1}|^2 + p_1 p_{-1} (Z_1 \overline{Z_{-1}} \exp(i2\eta  x_1)+ \overline{Z_1} Z_{-1} \exp(-i2\eta x_1)) \\
&= p_1^2 |Z_1|^2 + p_{-1}^2 |Z_{-1}|^2 + 2 p_1 p_{-1} |Z_1| |Z_{-1}|  \cos(2 \phi_1 + 2 x_1 \eta) \\
&= p_1^2 |Z_1|^2 + p_{-1}^2 |Z_{-1}|^2 + 2 p_1 p_{-1} |Z_1| |Z_{-1}| 
\\ 
& \qquad -2 p_1 p_{-1}  |Z_1| |Z_{-1}| (1-\cos(2 \phi_1 + 2 x_1 \eta)) \\
&= \Expect{|Z_{Y_1}|}^2 - 
4 p_1 p_{-1} |Z_1| |Z_{-1}| \sin^2 ( \phi_1 +  x_1 \eta),
\end{align*} 
Without loss of generality, assume that $|Z_1| \geq |Z_{-1}|$. There are two cases to consider.
The first case is that $Z_1$ and $Z_{-1}$ are comparable in  magnitude: $|Z_1| \leq 2 |Z_{-1}|$. In this case,  we can continue the bound by
\begin{align*} 
& \leq \Expect{|Z_{Y_1}|}^2 - 
2 p_1 p_{-1} |Z_1|^2 \sin^2 ( \phi_1 +  x_1 \eta) \\
& \leq \Expect{|Z_{Y_1}|}^2 (1 - 
2 \kappa (1-\kappa)  \sin^2 ( \phi_1 +  x_1 \eta) ) ,
\end{align*} 
since $1 \in J$. This gives
\begin{align*} 
&\left|p_1 \exp(i \eta  x_1) Z_1 + p_{-1} \exp(-i \eta  x_1 ) Z_{-1}\right|^{1+\nu} \\ 
& \leq \Expect{|Z_{Y_1}|}^{1+\nu}(1 - 
2 \kappa (1-\kappa)  \sin^2 ( \phi_1 +  x_1 \eta) )^{(1+\nu)/2}\\
& \leq \Expect{|Z_{Y_1}|^{1+\nu}}(1 - 
 \kappa (1-\kappa)  \sin^2 ( \phi_1 +  x_1 \eta) ),
\end{align*} 
since the map $\zeta \mapsto \zeta^{1+\nu}$ is convex.

The second case is when $|Z_1| > 2|Z_{-1}|$.
Recall that we have already shown
\begin{align*} 
&\left|p_1 \exp(i \eta  x_1) Z_1 + p_{-1} \exp(-i \eta  x_1 ) Z_{-1}\right|^2 \\ 
&= \Expect{|Z_{Y_1}|}^2 - 
4 p_1 p_{-1} |Z_1| |Z_{-1}| \sin^2 ( \phi_1 +  x_1 \eta) \\
& \leq \Expect{|Z_{Y_1}|}^2 . 
\end{align*}
Claim \ref{clm:MinPhi} implies that 
\begin{align*} 
&\left|p_1 \exp(i \eta  x_1) Z_1 + p_{-1} \exp(-i \eta  x_1 ) Z_{-1}\right|^{1+\nu} \\ 
& \leq \Expect{|Z_{Y_1}|}^{1+\nu} \\
& \leq (1-c_1 \nu)  \cdot \Expect{|Z_{Y_1}|^{1+\nu}} \\
& \leq (1-c_1 \nu \sin^2( \phi_j + x_j \eta ) )   \cdot \Expect{|Z_{Y_1}|^{1+\nu}}.
\end{align*} 

Finally, setting $c_0 = \min\{c_1, \kappa (1-\kappa)\}$, we get a bound that applies in both cases:
\begin{align*} 
%&\left|p_1 \exp(i \eta  x_1) Z_1 + p_{-1} \exp(-i \eta  x_1 ) Z_{-1}\right|^{1+\nu} \\
\left |\Ex{Y}{\exp({i \eta \ip{x}{Y}})} \right|^{1+\nu}
& \leq (1-c_0 \nu \sin^2( \phi_j + x_j \eta ) ) \cdot   \Expect{|Z_{Y_1}|^{1+\nu}}.
\end{align*}
This proves the base case of the induction and 
also allows to perform the inductive step.

\end{proof}

\subsection{An Average Direction} 
\label{sec:fourierboundAverageU}
In this section we analyze
the bound from the previous section for an average direction $X$
in a two-cube $A \subset \Z^n$.
This step has no analogy in the proof
of Theorem~\ref{thm:tech}.
To compute the expectation over an average direction, we reveal the entropy of $X$ coordinate by coordinate
in reverse order (from the $n$'th coordinate to the first one).

In analogy with $\gamma_1, \dotsc, \gamma_n$, define
the following functions $\mu_1,\ldots,\mu_n$.
For each $j \in [n]$, let
$$\mu_j(x) = \mu_j(x_{>j}) 
= \min_{\epsilon \in A_j} \Pr[X_j = \epsilon | X_{>j}= x_{>j}];$$
this is well-defined for $x$ in $A = \text{supp}(X)$.
In analogy \an{with} the definition of $J(y)$,
let 
$$J'(x) = \{j \in [n]: \mu_j(x) \geq \kappa\}.$$
In this section, we define the set $G$ differently, but use the same notation. Let 
$$ G(x,y) = G_{A,B,\kappa}(x,y) =  J'(x) \cap J(y).$$
Recall that $\gamma_j$, $\phi_j$
and $J(\cdot)$ depend on the set $B$,
on $y \in B$ and on $x \in \Z^n$.
In the following lemma, we fix an arbitrary $y \in B$, and take the expectation over a random $X \in A$.
We allow $G$ to be a random set that depends on $X$, and
$\phi_j$ to be a random variable
that depends on $X_{>j}$.

\begin{lemma} \label{lemma:techU2} 
For every $\kappa > 0$ and $0 < c_0 \leq 1$,  
		there is a constant $c > 0$ so that the following holds. For every $0 < \nu\leq 1$, every angle $\eta \in \R$, 
		every $B \subseteq \{\pm 1\}^n$,
		every $y \in B$,
%		 random variable $Y$ over $\{\pm 1\}^n$ and an independent 
every random variable $X$ taking values in a two-cube $A\subseteq \Z^n$ with differences $d_j = u_j - v_j$, 
		\begin{align*}
		\Ex{X}{\prod_{j \in J} (1-c_0 \nu \sin^2(\phi_j + X_j \eta))}^{1+\nu} & \leq \Ex{X}{\exp\Big(- c \nu \sum_{j \in G} \sin^2(d_j \eta) \Big)}.
		\end{align*}
\end{lemma}

\begin{proof}
The proof is by induction on $n$. 
Recall that
$\phi_j$ and $\mu_j$ is determined by $x_{>j}$.
In particular, whether or not $n \in G(x,y)$ 
does not depend on $x$.
If $n \notin G(x,y)$, the proof holds by induction, or is trivially true for $n=1$. 
So assume that $n \in G(x)$. 
Start with
\begin{align*}
 	& \Ex{X}{\prod_{j \in J} (1- c_0 \zeta \sin^2( \phi_j + X_j \eta ))} \\
%	 	& = \Ex{X_{n}}{(1- c_0 \zeta \sin^2( \phi_n + X_n \eta )) \Ex{X|X_{n}}{
%		\prod_{j \in J:j<n} (1- c_0 \zeta \sin^2( \phi_j + X_j \eta ))} }	\\
	 	& = \Ex{X_{n}}{(1- c_0 \zeta \sin^2( \phi_n + X_n \eta )) Z_{X_n}  },	\
		\end{align*}
where for $a \in A_n : = \{u,v\}$,
$$Z_a = \Ex{X|X_{n}=a}{
		\prod_{j \in J:j<n} (1- c_0 \sin^2( \phi_j + X_j \eta ))}.$$
		If $n=1$, then $Z_u = Z_v = 1$.
Assume without loss of generality 
		that $Z_u \geq Z_v$.
There are two cases to consider.
The first case is that $Z_u > 2 Z_v$.
In this case, Claim~\ref{clm:MinPhi} implies
\begin{align*}
 	 \Ex{X}{\prod_{j \in J} (1- c_0 \nu \sin^2( \phi_j + X_j \eta ))}^{1+\nu}  &\leq
		\Expect{Z_{X_n}}^{1+\nu} \\
		&\leq (1- c_1 \nu) \Expect{Z_{X_n}^{1+\nu}}\\
		&\leq \exp(- c_1 \nu) \Expect{Z_{X_n}^{1+\nu}}.
		\end{align*} 
%\begin{align*}
%		(p_u Z_u + p_v Z_v)^{1+\nu} 
%		& = Z_u^{1+\nu} (p_u + p_v \tfrac{Z_v}{Z_u})^{1+\nu} \\
%		& \leq Z_u^{1+\nu} (1-c_1)(p_u + p_v (\tfrac{Z_v}{Z_u})^{1+\nu}) \\
%		& \leq  (1-c \sin^2(d_n \eta) )(p_u Z_u^{1+\nu} + p_v Z_v^{1+\nu}) ,
%		\end{align*}
%		as long as $c \leq c_1$.
The second case is when $Z_u \leq 2 Z_v$. 
%For $a \in \{u,v\}$, define $p_a = \Pr[X_n = a]$. 
By Claim~\ref{clm:OneGoodChoice}, 
$$\max \big\{ |\sin(\phi_n + u \eta)| , |\sin(\phi_n + v \eta)|
\big\} \geq \tfrac{\sin(d_n \eta)}{2}.$$ 
Since $\mu_n(x) \geq \kappa$,
%If 
%$|\sin(\phi_n + u \eta)|  \geq \tfrac{\sin(d_n \eta)}{2}$ then
\begin{align*}
 &\Ex{X_{n}}{(1- c_0 \nu \sin^2( \phi_n + X_n \eta )) Z_{X_n}  }^{1+\nu} \\
 & \leq (\Ex{X_{n}}{Z_{X_n}  } - \kappa c_0 \nu \tfrac{\sin^2(d_n \eta)}{4} \tfrac{Z_{u}}{2})^{1+\nu}\\
 & \leq (\Ex{X_{n}}{Z_{X_n}} (1 -  \tfrac{\kappa  c_0 \nu}{8} \sin^2(d_n \eta)))^{1+\nu}\\
 & \leq \Ex{X_{n}}{Z^{1+\nu}_{X_n}}  \exp( - \tfrac{c_0 \kappa \nu}{8} \sin^2(d_n \eta)).
\end{align*}
% If $\sin(\phi_n + v_n \eta) \geq \sin(d_n \eta)/2$, we obtain the same bound with a similar proof. 

In both cases, 
$$ \Ex{X}{\prod_{j \in J} (1- c_0 \sin^2( \phi_j + X_j \eta ))}^{1+\nu} \leq \exp(-c \nu \sin^2(d_n \eta))  \Ex{X_n}{Z_{X_n}^{1+\nu}},$$ for some constant $c(\kappa,c_0) >0$.
This proves the base case of the induction and 
also allows to perform the inductive step.

\end{proof}

\subsection{Putting It Together}
\label{sec:mainPfU}
%We shall set the value of the constant $C$ during the proof. 
%\begin{align*}
%\Ex{\theta}{ \left |\Ex{Y}{ \exp(2\pi i \theta \cdot\ip{x}{Y})} \right|}
% \geq \rho \sqrt{\ln(1/\rho)} .
%\end{align*}

\begin{proof}[Proof of Theorem \ref{thm:mainTechU}]
Let $\mu >0$ and $0 < \nu \leq 1$
be so that 
  $$ \mu^{(1+\nu)^2}
  \geq 3  \exp(-\tfrac{ \nu n}{C}) +
 \tfrac{R_C(A)}{50 \sqrt{\nu}};$$
if no such $\mu,\nu$ exist then the theorem is trivially true.
Let $$A_0 = \Big\{ x \in A : 
\Ex{\theta}{ \Big |\Ex{Y}{ \exp(2\pi i \theta \cdot\ip{x}{Y})} \Big|}
 \geq \mu \Big\}.$$
%the constant $C$ is specified later on.
Denote the size of $A_0$ by $2^{\alpha n}$.
Assume towards a contradiction that 
$\alpha +\beta \geq 1+\delta$.
Let $X$ be uniformly distributed in $A_0$, independently of $Y$ and $\theta$.
Let $\lambda = \tfrac{\delta}{7}$,
and let $\kappa$ be as in~\eqref{1}.
By Lemma~\ref{lemma:techU},
		\begin{align*}
& \Ex{X,\theta}{\Big|\Ex{Y}{\exp({i 2\pi \theta \ip{x}{Y}})}\Big|}^{(1+\nu)^2} \\
& \leq \Ex{X,\theta}{\Big|\Ex{Y}{\exp({i 2\pi \theta \ip{x}{Y}})}\Big|^{1+\nu}}^{1+\nu} \\
& \leq \Ex{X,\theta}{\Ex{Y}{
	\prod_{j \in J} (1- c_0 \nu \sin^2( \phi_j + x_j 2 \pi \theta ) }
}^{1+\nu} .
\end{align*}
By Lemma~\ref{lemma:techU2}, we can continue
		\begin{align*}
& = \Ex{Y,\theta}{\Ex{X}{
	\prod_{j \in J} (1- c_0 \nu \sin^2( \phi_j + x_j 2 \pi \theta ) }}^{1+\nu} \\
& \leq \Ex{Y,\theta}{\Ex{X}{
	\prod_{j \in J} (1- c_0 \nu \sin^2( \phi_j + x_j 2 \pi \theta ) }^{1+\nu}} \\
& \leq \Ex{X,Y,\theta}{\exp (- c \nu D(\theta))} ,
\end{align*}
where
$$D(\theta) = D_{x,y}(\theta) = \sum_{j \in G(x,y)}  \sin^2( 2 \pi \theta d_j).$$
%
%		\begin{align*}
%\left |\Ex{X,\theta}{|\Ex{Y}{\exp({i 2\pi \theta \ip{x}{Y}})}|} \right|^{1+\nu} & \leq \Ex{X,Y,\theta}{\exp\Big(- c \nu \sum_{j \in G} \sin^2(\an{2 \pi d_j} \theta) \Big)}\\
%& \leq \Ex{X,Y,\theta}{\exp\Big(- c \nu D(\theta) \Big)},
%\end{align*}
%where here $D(\theta) = \sum_{j \in G(x,y)}  \sin^2( \an{2 \pi }\theta d_j)$.
%
By Lemma~\ref{lemma:encoding1}, $|J(y)| > n(\beta - 3 \lambda)$ for all but $2^{n(\beta - 2 \lambda)}$ choices for $y$. Similarly, $|J'(x)|> n(\alpha - 3 \lambda)$ for all but $2^{n(\alpha - 2 \lambda)}$  choices of $x$. 
By assumption,
$\beta - 3 \lambda + \alpha - 3 \lambda \geq \lambda$.
Since $|G(x,y)| \geq |J(y)| + |J'(x)| - n$, 
\begin{align*}
&\Pr [  |G(X,Y)| \leq \lambda n ]  \\
& \leq \Pr[|J(Y)| \leq n(\beta - 3 \lambda)] + \Pr[|J'(X)| \leq n(\alpha - 3 \lambda)]\\
& \leq 2^{-2\lambda n}+ 2^{-2\lambda n}.
\end{align*}

%By Lemma~\ref{lemma:encoding2},
%$$\Pr_X \Big[ \Pr_Y[ |G(x,Y)| \leq \lambda n ] \geq 2^{-\lambda n} \Big] 
%\leq 2^{n(1-\beta+6\lambda) - n(1-\beta+\delta)}
% = 2^{- \lambda n}.$$

\begin{claim*} 
Let $x,y$ be so that $G(x,y) \geq \lambda n$.
For every $0 \leq \rho\leq \tfrac{\lambda n}{4}$
and integer $\ell > 0$,
	$$\Pr_\theta  [D(\theta) < \rho ] \leq \frac{4 r_\ell(A)}{( \lambda n)^{2\ell+1/2}}   \sqrt{\rho} .$$
\end{claim*}
Given the claim, 
for every $x,y$ so that $G(x,y) \geq \lambda n$
and $\ell > 0$,
\begin{align*}
& \Ex{\theta}{ \exp (- c \nu D(\theta) ) } \\
& =
\int_0^1 \Pr_{\theta} [
 \exp (- c \nu D(\theta) ) > t  ] \, \mathrm{d} t
\\
& \leq \exp(- \tfrac{c \nu \lambda n}{4}) + 
\int_{\exp(- c \nu \lambda n /4)}^1 \Pr_{\theta} [
   D(\theta) < - \tfrac{\ln t}{c\nu} ] \, \mathrm{d} t
\\
& \leq \exp(-\tfrac{c \nu \lambda n}{4}) + 
\frac{4 r_\ell(A)}{(\lambda n)^{2\ell+1/2}}  \int_{0}^1 
 \sqrt{- \tfrac{ \ln t}{ c \nu}}
 \, \mathrm{d} t .
%\\
%&\leq  
%\int_0
%\tfrac{ 2^{2\ell+2}r_\ell}{n^{2\ell+1/2}} \int_{0}^{n/4}  \exp(- c \rho) \sqrt{\rho} \, \mathrm{d}\rho\\
%& \leq C \cdot \Big(\tfrac{2^{2\ell} r_\ell}{n^{2\ell+1/2}} + \exp(-n/C)\Big),
\end{align*}
The integral 
$\int_{0}^1 
 \sqrt{- \ln t}
 \, \mathrm{d} t \leq 1$ converges to a constant. 
For an appropriate $C = C(\beta,\delta) >0$
and $\ell >0$,
we get the desired contradiction.
 \begin{align*}
 \mu^{(1+\nu)^2}
 & \leq 2 \cdot 2^{-2 \lambda n} + 
 \exp(-\tfrac{c \nu \lambda n}{4}) +
\frac{4 r_\ell(A)}{\sqrt{c \nu} (\lambda n)^{2\ell+1/2}} \\
 & < 3  \exp(-\tfrac{ \nu n}{C}) +
 \tfrac{R_C(A)}{50 \sqrt{\nu}} .
 \end{align*}

\begin{proof}[Proof of Claim]
Let $G = G(x,y)$.
Observe that
\begin{align*}
& \Ex{\theta}{  ( |G|-2D(\theta) )^{2\ell} } \\
&= \Ex{\theta}{  \Big( \sum_{j \in G} \cos(4\pi d_j \theta)\Big)^{2\ell} }\\
&= 2^{-2\ell} \Ex{\theta}{   \Big( \sum_{j \in G} \exp(4\pi i d_j \theta)+ \exp(- 4\pi i  d_j \theta)\Big)^{2\ell} } \\
& \leq 2^{-2\ell} r_\ell(A) ;
\end{align*}
the last equality follows from the fact that of the $\leq (2|G|)^{2\ell}$ terms in the expansion, the only ones that survive are the ones with phase $0$. There are at most $r_\ell(A)$ such terms, and each contributes $1$. 

By Markov's inequality, since $|G| \geq \lambda n$,
\begin{align*}
\Pr_\theta  [D(\theta) \leq \tfrac{\lambda n}{4}  ] 
& \leq \Pr_\theta \Big [ (|G|-2D(\theta) )^{2\ell} \geq (\tfrac{\lambda n}{2})^{2\ell} \Big] \leq \frac{2^{-2\ell} r_\ell(A)}{(\lambda n/2)^{2\ell}} 
= \frac{r_\ell(A)}{(\lambda n)^{2\ell}}.
\end{align*}
This proves the claim for $\rho = \tfrac{\lambda n}{4}$.

It remains to prove the claim
for $\rho < \tfrac{\lambda n}{4}$.
This part uses Kemperman's theorem~\cite{kemperman}
from group theory (in fact Kneser's theorem~\cite{kneser}
for abelian groups suffices).
Think of $[0,1)$ as the group $\R/\Z$.
Let 
$$S_\rho =  \{\theta \in \R/\Z : D(\theta) \leq \rho \}.$$

We claim that the $m$-fold sum $S_\rho+S_\rho+\dotsb+S_\rho
\subseteq \R/\Z$
is contained in $S_{\rho m^2}$. Indeed, 
\begin{align*}
|\sin(\eta_1+ \eta_2)| & = |\sin(\eta_1) \cos(\eta_2) + \sin(\eta_2) \cos(\eta_1)|\\
&\leq |\sin(\eta_1)|+ |\sin(\eta_2)|,
\end{align*}
and so 
\begin{align*}
\sin^2(\eta_1+\dotsb+ \eta_m) & \leq  (|\sin(\eta_1)|+ \dotsb + |\sin(\eta_m)|)^2\\ &\leq m  (\sin^2(\eta_1)+ \dotsb + \sin^2(\eta_m)) .
\end{align*}
It follows that
\begin{align*}
 D(\theta_1+\theta_2 + \cdots + \theta_m)
& \leq m (D(\theta_1) + D(\theta_2)+\cdots + D(\theta_m)) \\
& \leq m^2 \max \{D(\theta_1) , D(\theta_2),
\ldots,D(\theta_m)\} .
\end{align*}

Kemperman's theorem thus implies that 
$$|S_{\rho m^2}| \geq
|S_\rho+\dotsb+S_\rho| \geq m |S_\rho|,$$
as long as $S_{\rho m^2}$ is not
all of $\R/\Z$.
%\footnote{Here $S_\rho$ is a finite union of intervals in the torus, so we do not even need the full generality of Kemperman's theorem.} 
%$m$ times the measure of $S_\rho$. 
%So, the measure of $S_{\rho m^2}$ is at least $m$ times the measure of $S_{\rho}$. 
Since
\begin{align*}
\Ex{\theta}{D(\theta)}
= \sum_{j \in G} \Ex{\theta}{\sin^2(2 \pi \theta d_j)}
= \frac{|G|}{2} ,
\end{align*}
we can deduce that $|S_{\lambda n/4}| = \Pr_{\theta}[D(\theta) \leq \tfrac{\lambda n}{4}]$
is strictly less than one.
Hence, $S_{\lambda n/4}$ is not the full group $\R/\Z$.
Setting $m$ to be the largest integer so that $m^2\rho \leq \tfrac{\lambda n}{4}$, we can conclude
\begin{align*}
\Pr_\theta   [D(\theta) \leq \rho  ] &\leq \tfrac{1}{m} \Pr_\theta   [D(\theta) \leq \rho m^2  ] 
\leq \tfrac{1}{m} \Pr_\theta   [D(\theta) \leq \tfrac{\lambda n}{4}  ] .
\qedhere \end{align*}
\end{proof}
\end{proof}

\subsubsection*{Acknowledgements}
We wish to thank James Lee, Oded Regev, Avishay Tal
and David Woodruff for helpful conversations.

%
%\bibliographystyle{abbrv}
%\bibliography{refGH} 

\appendix

\section{Strict Convexity} \label{sec:techclaim}
\begin{proof}[Proof of Claim \ref{clm:MinPhi}]
If $\alpha_1 =0$, then the claim is trivially true. So, 
assume that $\alpha_1 >0$.  Without loss of generality, we may also assume that $\kappa>0$ is small enough so that $4^\kappa > \exp(\kappa + \kappa^2)$. 

Let $p = \Pr[W = 1] \in [\kappa, 1-\kappa]$ and $\xi = \tfrac{\alpha_{-1}}{\alpha_1} \in [0,\tfrac{1}{2}]$. 
So, 
\begin{align*}
\frac{\Expect{\alpha_W}^{1+\nu}}{\Expect{\alpha_W^{1+\nu}}} 
%&= \frac{(p \alpha_1+ (1-p) \alpha_{-1})^{1+\nu}}{p \alpha_1^{1+\nu} + (1-p) \alpha_{-1}^{1+\nu}} \\
& = \frac{(p+ (1-p) \xi)^{1+\nu}}{p  + (1-p) \xi^{1+\nu}}. 
\end{align*}
We need to upper bound this ratio by $1-c_1 \nu$, for some constant $c_1$ that depends only on $\kappa$. 
Let
	$$\Phi(\xi,p,\nu)
	= (p+(1-p)\xi^{1+\nu}) - (p+(1-p)\xi)^{1+\nu}.$$
	We shall argue that there is a constant $c_1 = c_1(\kappa)>0$ such that $\Phi(\xi,p,\nu) \geq  c_1\nu$. This completes the proof, since 
	\begin{align*}
	&\frac{(p + (1-p) \xi)^{1+\nu}}{(p + (1-p) \xi^{1+\nu})} = 1 - \frac{\Phi(\xi,p,\nu)}{(p + (1-p) \xi^{1+\nu})} < 1-c_1\nu.
	\end{align*}

First,	we show that for every $\nu$ and $\xi$, 
	the function $\Phi(\xi,p,\nu)$ is minimized when $p = \kappa$. Consider
	\begin{align*}
	\frac{\partial \Phi}{\partial p} & = 1 - \xi^{1+\nu} - (1+\nu) (p+(1-p)\xi)^{\nu} (1-\xi) \\
	&\geq 1- \xi^{1+\nu}- (1+\nu) (1-\xi)\\
%	& \geq (1+\nu) \xi - \xi^{1+\nu} \\
	&\geq \xi (1+\nu - \xi^{\nu})>0,
	\end{align*}
	since $\xi^{\nu}<1$.
	So, the minimum is achieved when $p = \kappa$.

Second, we claim that for every $\nu$ and $p$, 
	the function $\Phi(\xi,p,\nu)$ is minimized when $\xi = \tfrac{1}{2}$. Consider
	\begin{align*}
	\frac{\partial \Phi}{\partial \xi} & = (1-p)(1+\nu)\xi^{\nu} - (1+\nu)(p+(1-p)\xi)^{\nu} (1-p)\\
	&= (1-p)(1+\nu) (\xi^\nu - (p+ (1-p)\xi)^\nu) <0,
	\end{align*}
	since $p+ (1-p)\xi > \xi$. So, the minimum is achieved when $\xi = 1/2$. 

Third, we control the derivative with respect to $\nu$
for $\xi = \tfrac{1}{2}$ and $p = \kappa$.
Consider
	\begin{align*}
	\frac{\partial  \Phi}{\partial \nu}(\tfrac{1}{2},\kappa,\nu) &=  (1-\kappa) \ln(\tfrac{1}{2}) (\tfrac{1}{2})^{1+\nu} - \ln (\tfrac{1+\kappa}{2}) (\tfrac{1+\kappa}{2})^{1+\nu} \\
%	& = (\tfrac{1}{2})^{1+\nu} ((1-\kappa) \ln(\tfrac{1}{2}) - \ln (\tfrac{1+\kappa}{2}) (1+\kappa)^{1+\nu}) \\
	& \geq (\tfrac{1}{2})^2 ((1-\kappa) \ln(\tfrac{1}{2}) - \ln (\tfrac{1+\kappa}{2}) (1+\kappa)^{1+\nu}),  
	\end{align*}since $\nu \leq 1$. The expression 
	$$(1-\kappa) \ln(\tfrac{1}{2}) - \ln (\tfrac{1+\kappa}{2}) (1+\kappa)^{1+\nu}$$ 
	only increases with $\nu$. When $\nu=0$, this expression is
	\begin{align*}
%&	(1-\kappa) \ln(\tfrac{1}{2}) - \ln (\tfrac{1+\kappa}{2}) (1+\kappa)\\
%	& = 	\ln(\tfrac{1}{2}) -\kappa \ln(\tfrac{1}{2}) - 
%	 \ln (1+\kappa) (1+\kappa) 
%- \ln (\tfrac{1}{2}) -\kappa \ln (\tfrac{1}{2})\\
% & = 
\ln(\tfrac{2^{2\kappa}}{(1+\kappa)^{1+\kappa}}) \geq  \ln(\tfrac{4^{\kappa}}{\exp(\kappa(1+\kappa))})>0,
	\end{align*}
	since $4^\kappa > \exp(\kappa + \kappa^2)$. This proves that $\frac{\partial  \Phi}{\partial \nu}(\tfrac{1}{2},\kappa,\nu)>c_1$ for some constant $c_1(\kappa) >0$. 
  
Finally,
   $$\Phi(\xi,p,\nu) \geq \Phi(\tfrac{1}{2},\kappa,\nu)
   = \int_0^\nu 	\frac{\partial  \Phi}{\partial \nu}(\tfrac{1}{2},\kappa,\zeta) \, \mathrm{d} \zeta
   \geq \int_0^\nu c_1 \, \mathrm{d} \zeta 
   = c_1 \nu .\qedhere $$
%   So, we must have that $$\Phi(1/2,\kappa,\nu)\geq \Phi(1/2,\kappa,0) + \frac{\partial \Phi}{\partial \nu}(1/2,\kappa,0) \cdot \nu = \frac{\partial \Phi}{\partial \nu}(1/2,\kappa,0) \cdot \nu.$$
\end{proof}

\end{document}